\documentclass[final,1p,times]{elsarticle}

\usepackage{amssymb}
\usepackage{amsthm}
\usepackage{hyperref}
\usepackage{amsmath,amssymb,amsopn,amsfonts,mathrsfs,amsbsy,amscd}
\usepackage{longtable}
\journal{}

\newtheorem{definition}{Definition}[section]
\newtheorem{theorem}{Theorem}[section]
\newtheorem{proposition}{Proposition}[section]
\newtheorem{lemma}{Lemma}[section]

\numberwithin{equation}{section}

\begin{document}
\bibliographystyle{amsplain}
\begin{frontmatter}
\title{Biderivations, local and 2-local derivation and automorphism of simple $\omega$-Lie algebras}


\author[label2]{Hassan Oubba}
\address[label2]{Moulay Ismail University of Meknes\\ Faculty of Sciences Meknes
B.P. 11201 Zitoune Meknes, Morocco\\
{hassan.oubba@edu.umi.ac.ma}}

\begin{abstract}
 Given a finite-dimensional complex simple $\omega$-Lie algebras $\mathfrak{}$ over $\mathbb{C}$.  We prove that every local ,$2-$local derivation is a derivation and every local (resp. 2-local) automorphisms  are  automorphisms or an anti-automorphis (resp. automorphism). We characterize also biderivation, $\frac{1}{2}$-derivation and local (2-local) $\frac{1}{2}$-derivation of $\mathfrak{g}$. 
\end{abstract}
\end{frontmatter}

\section{Introduction} \label{Introduction}
The history of local mappings begins with the Gleason-Kahane-Zelazko theorem in \cite{gle} and \cite{kah}, which is a fundamental contribution to the theory of Banach algebras. This theorem asserts that every unital linear functional $F$ on a complex unital Banach algebra $A$, such that $F(a)$ belongs to the spectrum $\sigma(a)$ of $a$ for every $a \in A$, is multiplicative.

In the past two decades, a well-known and active direction in the study of the automorphism (respectively, derivation) of Lie is the problem about local automorphisms (respectively, local derivations) and 2-local automorphisms (respectively, 2-local derivations). In \cite{Cos} the autors prove that any local automrphism of  a finite dimensional simple Lie algebra over an
algebraically closed field $\mathbb{K}$ of characteristic zero is an automorphism or an anti-automorhism. In, \cite{loc} it was proved that any 2-local automorphism of finite-dimensional simple Lie algebra over an
algebraically closed field $\mathbb{K}$ of characteristic zero is an automorphism. Similar
results have been obtained in \cite{Ayo2, Ayo3} for local and  2-local derivation of a finite-dimension Lie algebras. It is natural to study corresponding generalizations of automorphisms or derivations of $\omega$-Lie algebras.

Subsequently, in \cite{kad}, the concept of local derivation is introduced, and it is proved that each continuous local derivation from a von Neumann algebra into its dual Banach bimodule is a derivation. Numerous new results regarding the description of local derivations and local automorphisms of algebras have since emerged (see, for example, \cite{oub, kad, Ayo1, Ayo2, sem}).

$\omega$-Lie algebras are a natural generalization of Lie algebras. They were introduced by Nurowski in \cite{nur1}, motivated by the study of isoparametric hypersurfaces in Riemannian geometry. There have been extensive works on $\omega$-Lie algebras (see \cite{nur2, Chen0, Chen1, Chen2} and references therein). In the cases of dimensions $1$ and $2$, there are no nontrivial $\omega$-Lie algebras. The first example of a nontrivial $3$-dimensional $\omega$-Lie algebra was given by Nurowski \cite{nur1}, where the author provided a classification of $3$-dimensional $\omega$-Lie algebras over the field of real numbers under the action of the $3$-dimensional orthogonal group.

On the other hand, the notion of biderivations has appeared in different areas. Maksa used biderivations to study real Hilbert space \cite{Mak}. Vukman investigated symmetric biderivations in prime and semiprime rings \cite{Vuk}. The well-known result that every biderivation on a noncommutative prime ring $A$ is of the form $\lambda[x, y]$ for some $\lambda$ belonging to the extended centroid of $A$ was discovered independently by Bresar et al \cite{Bresar}, Skosyrskii \cite{Sko}, and Farkas and Letzter \cite{Far}. Biderivations were connected with noncommutative Jordan algebras by Skosyrskii and with Poisson algebras by Farkas and Letzter. Besides their wide applications, biderivations are interesting in their own right and have been introduced to Lie algebras \cite{Wang}, which have been studied by many authors recently. In particular, biderivations are closely related to the theory of commuting linear maps, which has a long and rich history. For the development of commuting maps and their applications, we refer to the survey \cite{br}. It is worth mentioning that Bresar and Zhao considered a general but simple approach for describing biderivations and commuting linear maps on a Lie algebra $L$ having their ranges in an $L$-module \cite{brz}. This approach covered most of the results in \cite{br, brz, oub, Mak, Vuk, Bresar} and inspires us to generalize their method to Hom-Lie algebras.

In \cite{oubb}, we give a description of local (2-local) derivations and automorphisms and biderivations of three dimensional complex $\omega$-Lie algebras in this paper we give similar result for four dimensional complex $\omega$-Lie algebras

In this paper, the organization is as follows. In Section 2, we recall some definitions and results needed for this study and we prove that any symmetric derivation of semisimple Lie algebra is trivial. In Section 4, we investigate local and $2$local derivations on finite-dimensional semisimple $\omega$-Lie algebras $\mathfrak{g}$ over $\mathbb{C}$, we prove that every local and 2-local derivation of $\mathfrak{g}$ are derivations. Section 5 is devoted to local and $2$-local automorphisms on $\mathfrak{g}$, we prove that any local automorphism on $\mathfrak{g}$ is is an automrphism or an anti-automrphism and any 2-local automorphism on $\mathfrak{g}$ is an automorphism. In section 6, we characterize all biderivations of simple $\omega$-Lie algebras and we give the chacterization of all biderivation  on $4$-dimensional $\omega$-Lie algebras. In the last section we give a description of $\frac{1}{2}$-derivation of simpl $\omega$-Lie algebras. 
\section{Biderivations of semisimple Lie algebras}
\begin{definition}
    For a Lie algebra $(\mathfrak{g}, [\cdot,\cdot])$, a \emph{derivation} is a linear map $D\in End(\mathfrak{g})$ such that 
    $$D([x,y])=[D(x),y]+[x,D(y)],\quad \forall x,y \in \mathfrak{g}.$$
\end{definition}
\begin{definition}
For a Lie algebra $(\mathfrak{g}, [\cdot,\cdot])$, a \emph{biderivation} is a bilinear map $\delta\in\mathrm{Bil}(\mathfrak{g},\mathfrak{g};\mathfrak{g})$ satisfying the following properties:
$$\delta([x,y],z)=[x,\delta(y,z)]+[\delta(x,z),y],$$
 $$\delta(x,[y,z])=[y,\delta(x,z)]+[\delta(x,y),z],$$
 for all $x,y,z \in \mathfrak{g}$.
\end{definition}
 Let $BDer(\mathfrak{g})$ represent the collection of all biderivations defined on $\mathfrak{g}$. It is evident that this set forms a vector space.\\
  A biderivation $\delta \in BDer(\mathfrak{g})$ is termed symmetric if $\delta(x,y)=\delta(y,x)$ holds for all $x,y \in \mathfrak{g}$; conversely, it is designated as skew-symmetric if $\delta(x,y)=-\delta(y,x)$ for all $x,y \in \mathfrak{g}$. We denote the subspaces of all symmetric biderivations and all skew-symmetric biderivations on $\mathfrak{g}$ as $BDer_+(\mathfrak{g})$ and $BDer_-(\mathfrak{g})$, respectively.
  
We recall new the definition of  anti-commuting (resp. commuting) linear map and some relation ship between it and the the biderivation of Lie algebra. Using this relation ship we show that any symmetric biderivation on quadratic Lie algebra such that $Der(\mathfrak{g})=ad(\mathfrak{g})$ and $Z(\mathfrak{g})=\lbrace 0\rbrace$  is trivial.
\begin{definition}
 Let $(\mathfrak{g},[\,,\,])$ be a Lie algebra and $\alpha : \mathfrak{g} \rightarrow \mathfrak{g}, x \mapsto \alpha(x)$ is a linear map.
 \begin{enumerate}
  \item $\alpha$ is called an anti-commuting map on the Lie algebra $(\mathfrak{g},[\,,\,])$ if 
  $$[\alpha(x),y]= \, [\alpha(y),x],\quad \forall x, y \in \mathfrak{g};$$
  \item $\alpha$ is called a commuting map on $(\mathfrak{g},[\,,\,])$ if 
  $$[\alpha(x),y]= \, [x,\alpha(y)],\quad \forall x, y \in \mathfrak{g};$$
\item $\alpha$ is in the centroid of  $(\mathfrak{g} ,[\,,\,])$ if 
  $$\alpha([x,y])= \, [\alpha (x),y],\quad \forall x, y \in \mathfrak{g}.$$
 \end{enumerate} 
\end{definition}
The centroid of the Lie algebra $(\mathfrak{g},[\,,\,])$ will be noted $\mbox{Cent}(\mathfrak{g})$, the set of all commuting map on $\mathfrak{g}$ will be noted by $Cmap(\mathfrak{g})$ and the set of anti-commuting map on the Lie algebra $(\mathfrak{g},[\,,\,])$ will be noted $\mbox{Acmap}(\mathfrak{g})$ and the set of all will be noted and the set of commuting map on the Lie algebra $(\mathfrak{g},[\,,\,])$ will be noted $\mbox{Cmap}(\mathfrak{g}).$\\

The proof of the following proposition is obvious.
\begin{proposition}\label{hom}
	\begin{enumerate}	
		\item 	Let $(\mathfrak{g},[,])$ be a Lie algebra with an anti-commuting map $\alpha$.
		 Then the bilinear map $\delta$ given by $\delta(x,y):=[\alpha(x),y]$ is a symmetric biderivation of $\mathfrak{g}$.
		\item If $(\mathfrak{g},[,])$ is a Lie algebra such that $Der(\mathfrak{g})=ad(\mathfrak{g})$, we have the following converse:
        \begin{enumerate}
            \item Let $\delta :\mathfrak{g} \times \mathfrak{g} \rightarrow \mathfrak{g}, (x,y) \mapsto \delta(x , y)$  be a skew-symmetric biderivation of $\mathfrak{g}$. Then, there exist a commuting map $\alpha$  on Lie algebra $(\mathfrak{g},[,])$ such that 
		$$ \delta(x, y )=[\alpha(x),y], \quad \forall x,y \in \mathfrak{g}$$.
        \item Let $\delta :\mathfrak{g} \times \mathfrak{g} \rightarrow \mathfrak{g}, (x,y) \mapsto \delta(x , y)$  be a  symmetric biderivation of $\mathfrak{g}$. Then, there exist an anti-commuting map $\alpha$  on Lie algebra $(\mathfrak{g},[,])$ such that 
		$$ \delta(x, y )=[\alpha(x),y], \quad \forall x,y \in \mathfrak{g}$$.
        \end{enumerate}
	\end{enumerate}
\end{proposition}
\begin{proposition}\label{Acmap}
    Let  $(\mathfrak{g},[,], B)$ be a semisimple Lie algebra over a field of characteristic  $\neq 2$. Then, $ Acmap(\mathfrak{g})=\lbrace 0 \rbrace$ 
\end{proposition}
\begin{proof}
  Let $(\mathfrak{g},[,])$ be a semisimple Lie algebra and $B$ its Killing form. For any $\alpha \in Acmap(\mathfrak{g})$, we have  for any $x,y,z\in \mathfrak{g}$  
  \begin{eqnarray*}
      B( [x, \alpha(y)],z )&=&B( x,[\alpha( y),z] ) \\
      &=&B(x, [\alpha(z),y] )\\
      &=&B( [x,\alpha(z)], y )\\
      &=&B( [z,\alpha(x)], y )\\
      &=&B( z,[\alpha(x), y] )\\
      &=&B( z,[\alpha(y), x] )\\
      &=&-B( z,[x,\alpha(y)] )\\
    \end{eqnarray*}
    Then, $2B [x, \alpha(y)],z )=0$, for all $x,y,z\in \mathfrak{g}$. Since, $B$ is nondegenerated then $[x,\alpha(y)]=0$ for all $x,y\in \mathfrak{g}$. Therfore, $\alpha(y)\in Z(\mathfrak{g})=\lbrace 0 \rbrace$. Which ends the proof.
\end{proof}
\begin{theorem}
    Let $(\mathfrak{g},[,])$ be a semisimple Lie algebra over a fiel of characteristic $\neq2$. Then any symmetric biderivation of $\mathfrak{g}$ is the zero mapping.
\end{theorem}
\begin{proof}
   Let $\delta$ be a symmetric biderivation of a semisimple Lie algebra $\mathfrak{g}$, then by Proposition \ref{hom}, there exists $\alpha\in Acmap(\mathfrak{g})$ such that $\delta(x,y)=[\alpha(x),y]$ for all $x,y \in \mathfrak{g}$. Thinks, to Proposition \ref{Acmap}, we have $[\alpha(x),y]=0$, for all $x,y \in \mathfrak{g}$. Therefore, $\delta(x,y)=0$ for all $x,y \in \mathfrak{g}$. 
\end{proof}
\begin{lemma}\label{l1}
    Let $\mathfrak{s}=\mathfrak{s}_1\oplus \dots \oplus \mathfrak{s}_k$ be the decomposition into simple factors of a semisimple Lie algebra and let $f$ be a map in the centroid of $\mathfrak{s}$. There exist $\lambda_1, \dots, \lambda_k \in \mathbb{C}$ such that, if $x=\sum_{i=1}^kx_j$ with $x_j \in \mathfrak{s}_j$, then $f(x)=\sum_{i=1}^k \lambda \lambda_i x_i$.
\end{lemma}
\begin{proof}
    First notice that the result holds for $k=1$ because if $\mathfrak{s}$ is simple and $\lambda \in \mathbb{C}$ is an eigenvalue of $f$ then one has that $Ker(D-\lambda I)\neq \lbrace 0 \rbrace$ is an ideal of $\mathfrak{s}$ because $f(x)=\lambda x$ implies $f([x,y])=[f(x),y]=[\lambda x,y]$ and, therefore, it must be the whole $\mathfrak{s}$.\\
    Now, if $k \geq 2$, take $x_j \in \mathfrak{s}_j$ for some $j \leq k$. For all $i \leq k$ and $x \in \mathfrak{s}_i$, one has $0=f([x_i,x_j])=[x_i,f(x_j)]$, which clearly implies $f(x_j)\in \mathfrak{s}_j$ and the restriction $f_j=f_{\vert_{\mathfrak{s}_j}}$ is in the centroid of $\mathfrak{s}_j$ and, hence, there exists $\lambda_j\in \mathbb{C}$ such that $f_j=\lambda_j I$. 
\end{proof}
\begin{lemma}
\label{comcent}
Let $(\mathfrak{g},[\,,\,])$ be a semisimole Lie algebra. Then $\mbox{Cmap}(\mathfrak{g})= \mbox{Cent}(\mathfrak{g})$.
 \end{lemma}
 \begin{proof}
  Let $f \in \mbox{Cmap}(\mathfrak{g}).$ Then,  the bilinear map $\delta: \mathfrak{g}\times \mathfrak{g}\rightarrow \mathfrak{g},  $ defined by $\delta(x,y)= [f(x),y]),  \forall x, y \in \mathfrak{g},$ is a skew-symmetric biderivation of $(\mathfrak{g},[\,,\,])$. It follows, by Theorem 3.8 of \cite{TMC}, that there exists $\gamma\in \mbox{Cent}(\mathfrak{g})$ such that $[f(x),y]= \gamma([x,y])= [\gamma(x),y],\, \forall x, y \in \mathfrak{g}.$ Therefore $f= \gamma$ because $(\mathfrak{g},[\,,\,])$ is a centerless Lie algebra.
 \end{proof}
 \begin{theorem}
     Let $\mathfrak{s}=\mathfrak{s}_1\oplus \dots \oplus \mathfrak{s}_k$ be the decomposition into simple factors of a semisimple Lie algebra and let $\delta$ be a skew-biderivation of $\mathfrak{s}$. There exist $\lambda_1, \dots, \lambda_k \in \mathbb{C}$ such that, 
     $$\delta(x,y)=\sum_{i=1}^k \lambda_i[x_i,y_i],$$
     for all $x=\sum_{i=1}^kx_i, y=\sum_{i=1}^ky_i \in \mathfrak{s}$.
 \end{theorem}
 \begin{proof}
     Let $\delta$ be a skew-symmetric biderivation of $\mathfrak{s}=\mathfrak{s}_1\oplus \dots \oplus \mathfrak{s}_k$, then by Proposition \ref{hom} there exist a commuting linear map $f$ such that $\delta(x,y)=[f(x),y]$ for any $x,y \in \mathfrak{s}$. Therefore, by Lemma\ref{l1} and Lemma \ref{comcent} there exist $\lambda_1,\dots,\lambda_k$ such that
     $$\delta(x,y)=\sum_{i=1}^k \lambda_i[x_i,y_i].$$
 \end{proof}
 \begin{theorem} \label{simple}
     Let $\mathfrak{g}$ be a simple Lie algebra. Then, $\delta$ is a biderivation of $\mathfrak{g}$ if and only if there exists a scalar $\lambda$ such that 
     $$\delta(x,y)=\lambda[x,y], \quad \forall x,y \in \mathfrak{g}.$$
 \end{theorem}
\section{ $\omega$-Lie algebras}
\begin{definition}
A vector space over $\mathbb{C}$ is called an $\omega$-Lie algebra if there is a bilinear map $[,]: \mathfrak{g} \times \mathfrak{g} \rightarrow \mathfrak{g}$ and a skew-symmetric bilinear form $\omega: \mathfrak{g} \times \mathfrak{g} \rightarrow \mathbb{C}$ such that 
\begin{enumerate}
\item $[x,y]=-[y,x]$,
\item $[[x,y],z]+[[y,z],x]+[[z,x],y]=\omega(x,y)z+\omega(y,z)x+\omega(z,x)y$,
\end{enumerate}
for all $x,y,z \in \mathfrak{g}$.
\end{definition}
Clearly, an $\omega$-Lie algebra $\mathfrak{g}$ with $\omega=0$ is a Lie algebra, which is called a trivial $\omega$-Lie algebra. Otherwise, $\mathfrak{g}$ is called a nontrivial $\omega$-Lie algebra
\begin{definition}
Let $(\mathfrak{g},[,])$ be an $\omega$-Lie algebra. A linear map $D : \mathfrak{g} \rightarrow \mathfrak{g}$ is called derivation, if 
$$D([x,y])=[D(x),y]+[x,D(y)]$$
for all $x,y \in \mathfrak{g}$.
\end{definition}
We write $gl(\mathfrak{g})$ for the general linear Lie algebra on $\mathfrak{g}$. Then the set $Der(\mathfrak{g})$ of all derivations of $\mathfrak{g}$ forms a Lie subalgebra of $gl(\mathfrak{g})$, which is called the derivation algebra of $\mathfrak{g}$.
\begin{definition}
Let $(\mathfrak{g},[,])$ be an $\omega$-Lie algebra. A derivation $D : \mathfrak{g} \rightarrow \mathfrak{g}$ of $\mathfrak{g}$ is called $\omega$-derivation, if 
$$\omega(D(x),y)+\omega(x,D(y))=0$$
for all $x,y \in \mathcal{A}$.
\end{definition}
We write $Der_\omega(\mathfrak{g})$ for the set consisting of all $\omega$-derivations of $\mathfrak{g}$. Clearly,$Der_\omega(\mathfrak{g}) \subseteq Der(\mathfrak{g})$.
\begin{definition}
A linear map $\Delta: \mathfrak{g} \rightarrow \mathfrak{g}$ is said to be a local derivation if for any $x \in \mathfrak{g}$ there exists a derivation $D_x$ of $\mathfrak{g}$ such that $\Delta(x)=D_x(x)$.
\end{definition}
\begin{definition}
A (not necessary linear) map $\Delta: \mathfrak{g} \rightarrow \mathfrak{g}$ is said to be a $2$- local derivation if for any $x,y \in \mathfrak{g}$ there exists a derivation $D_{x,y}$ of $\mathfrak{g}$ such that $$\Delta(x)=D_{x,y}(x) \quad and \quad \Delta(y)=D_{x,y}(y).$$
\end{definition}

Let us initiate the categorization of 3-dimebsional and 4-dimensional nontrivial $\omega$-Lie algebras over $\mathbb{C}$.
\begin{theorem}(Chen-Liu-Zhang \cite{Chen1}). 
Let $\mathfrak{g}$ be a 3-dimensional nontrivial $\omega$-Lie algebra over $\mathbb{C}$, then it must be isomorphic to one of the following algebras:

(1) $L_1$: $[e_1,e_3] = 0$, $[e_2,e_3] = e_3$, $[e_1, e_2] = e_2$ and $\omega(e_2,e_3) = \omega(e_1,e_3) = 0$, $\omega(e_1, e_2) = 1$.

(2) $L_2$: $[e_1, e_2] = 0$, $[e_1,e_3] = e_2$, $[e_2,e_3] = e_3$ and $\omega(e_1, e_2) = 0$, $\omega(e_1,e_3) = 1$, $\omega(e_2,e_3) = 0$.

(3) $A_\alpha$: $[e_1, e_2] = e_1$, $[e_1,e_3] = e_1 + e_2$, $[e_2,e_3] = e_3 + \alpha e_1$ and $\omega(e_1, e_2) = \omega(e_1,e_3) = 0$, $\omega(e_2,e_3) = -1$, where $\alpha \in \mathbb{C}$.

(4) $B$: $[e_1, e_2] = e_2$, $[e_1,e_3] = e_2 + e_3$, $[e_2,e_3] = e_1$ and $\omega(e_1, e_2) = \omega(e_1,e_3) = 0$, $\omega(e_2,e_3) = 2$.

(5) $C_\alpha$: $[e_1, e_2] = e_2$, $[e_1,e_3] = \alpha e_3$, $[e_2,e_3] = e_1$ and $\omega(e_1, e_2) = \omega(e_1,e_3) = 0$, $\omega(e_2,e_3) = 1 + \alpha$, where $0, -1, \alpha \in \mathbb{C}$.
\end{theorem}
\begin{theorem}(Chen-Zhang \cite{Chen0}). Any $4$-dimensional nontrivial $\omega$-Lie algebra over $\mathbb{C}$ must be isomorphic to one of the following algebras:\\
$\bullet\,L_{1,1}:\,   [e_1,e_2]=e_2,  [e_2,e_3]=e_3, [e_4,e_2]=-y,  \omega(e_1,e_2)=1. $\\
$\bullet\,L_{1,2} :\, [e_1,e_2]=e_2, [e_2,e_3]=e_3, [e_4,e_1]=e_3, [e_4,e_2]=-e_4, \omega(e_1,e_2)=1.$\\
$\bullet\,L_{1,3} :\, [e_1,e_2]=e_2, [e_2,e_3]=e_3, [e_4,e_1]=e_2, [e_4,e_2]=-e_4, \omega(e_1,e_2)=\omega(e_4,e_1)=1.$\\
$\bullet\,L_{1,4} :\, [e_1,e_2]=e_2, [e_2,e_3]=e_3, [e_4,e_1]=e_2+e_3,[e_4,e_2]=-e_4,  \omega(e_1,e_2)=\omega(e_4,e_1)=1.$\\
$\bullet\,L_{1,5} :\, [e_1,e_2]=e_2, [e_2,e_3]=e_3, [e_4,e_1]=e_4, [e_4,e_2]=-e_4, \omega(e_1,e_2)=1. $\\
$\bullet\,L_{1,6}:\,[e_1,e_2]=e_2,\, [e_2,e_3]=e_3, [e_4,e_1]=e_4+e_2, [e_4,e_2]=-e_4, \omega(e_1,e_2)=\omega(e_4,e_1)=1. $\\
$\bullet\,L_{1,7} :\, [e_1,e_2]=e_2, [e_2,e_3]=e_3,  [e_4,e_1]=e_4,  [e_4,e_2]=e_3-e_4, \omega(e_1,e_2)=1.$\\
$\bullet\,L_{1,8} :\,[e_1,e_2]=e_2, [e_2,e_3]=e_3, [e_4,e_1]=e_4+e_2, [e_4,e_2]=e_3-e_4, \omega(e_1,e_2)=\omega(e_4,e_1)=1. $\\
$\bullet\,L_{2,1} :\, [e_1,e_3]=e_2, [e_2,e_3]=e_3, [e_4,e_2]=-e_4, \omega(e_1,e_3)=1.$\\
$\bullet\,L_{2,2} :\, [e_1,e_3]=e_2, [e_2,e_3]=e_3, [e_4,e_2]=-e_4, [e_4,e_1]=e_3,  \omega(e_1,e_3)=1. $\\
$\bullet\,L_{2,3}:\, [e_1,e_3]=e_2, [e_2,e_3]=e_3, [e_4,e_2]=-e_4, [e_4,e_1]=e_4,  \omega(e_1,e_3)=1. $\\
$\bullet\,L_{2,4} :\, [e_1,e_3]=e_2, [e_2,e_3]=e_3, [e_4,e_2]=-e_4, [e_4,e_1]=e_4+e_3,  \omega(e_1,e_3)=1. $\\
$\bullet\,E_{1,\alpha}(\alpha \neq 0,1) :\, [e_1,e_2]=e_2, [e_2,e_3]=e_3, [e_4,e_2]=-e_4, [e_4,e_1]=\alpha e_4,  \omega(e_1,e_2)=1. $\\
$\bullet\,F_{1,\alpha}(\alpha \neq 0,1) :\,[e_1,e_2]=e_2, [e_2,e_3]=e_3, [e_4,e_2]=-e_4, [e_4,e_1]=\alpha e_4+e_2,\\
\hspace*{2.5 cm} \omega(e_1,e_2)=\omega(e_4,e_1)=1. $\\
$\bullet\,G_{1,\alpha} :\,[e_1,e_2]=e_2, [e_2,e_3]=e_3, [e_4,e_2]=e_1-e_4, [e_4,e_1]=e_4+\alpha e_2, \\
\hspace*{2.5 cm} \omega(e_1,e_2)=1,  \omega(e_4,e_1)=\alpha . $\\
$\bullet\,H_{1,\alpha}:\,[e_1,e_2]=e_2, [e_2,e_3]=e_3, [e_4,e_2]=e_1+e_3-e_4,\, [e_4,e_1]=e_4+\alpha e_2, \\
\hspace*{2.5 cm} \omega(e_1,e_2)=1,  \omega(e_4,e_1)=\alpha. $\\
$\bullet\,\tilde{A_\alpha} :\, [e_1,e_2]=e_1,  [e_1,e_3]=e_1+e_2,  [e_2,e_3]=e_3+\alpha e_1, [e_4,e_3]=e_4, \omega(e_2,e_3)=-1. $\\
$\bullet\,\tilde{B} :\, [e_1,e_2]=e_2,[e_1,e_3]=e_2+e_3, [e_2,e_3]=e_1, [e_4,e_2]=-e_4, [e_4,e_1]=-2e_4, \omega(e_2,e_3)=2. $\\
$\bullet\,\tilde{C_\alpha} (\alpha \neq 0,-1) :\, [e_1,e_2]=e_2, [e_1,e_3]=\alpha e_3,  [e_2,e_3]=e_1, [e_4,e_1]=-(1+\alpha)e_4,
 \\ \hspace*{2.5 cm} \omega(e_2,e_3)=1+\alpha.$
\end{theorem}
\begin{theorem}[\cite{11}, Theorem 2] A finite-dimensional semisimple $\omega$-Lie algebra is
 either a Lie algebra, or has dimension $\leq 4$.
\end{theorem}
\begin{proposition}[\cite{12}, Proposition 7.1]
    The algebras $A_\alpha (\alpha \in \mathbb{C}$, $B$, and $C_\alpha\, (0,-1 \neq \alpha \in \mathbb{C})$ are all non-Lie 3-dimensional complex simple $\omega$-Lie algebras.
\end{proposition}
\begin{proposition}[\cite{12}, Proposition 7.2]
   There do not exist non-Lie 4-dimensional complex simple $\omega$-Lie algebras. 
\end{proposition}
The above results show essentially the importance of characterize biderivation, local (2-local) derivations and automorphism of low-dimensional $\omega$-Lie algebras. 
\section{Local and $2$-local derivations}
The exploration of local and $2$-local derivations of $\omega$-Lie algebras $\mathfrak{g}$ relies heavily on a key theorem established by Chen Y. and al., as documented in their work \cite{Chen2}. This theorem serves as our principal instrument for delving into the intricacies of the derivations in question. 
\begin{theorem}
Every local derivation of $L_{1,1}$ is a derivation.
\end{theorem}
\begin{proof}
Let $\Delta$ be an arbitrary local derivation of $L_{1,1}$. By the definition for all $x \in L_{1,1}$ there exists a derivation $D_x$ on $L_{1,1}$ such that $\Delta(x)=D_x(x)$.\\
According to Tble 3 in \cite{Chen2}, the derivation $D_x$ has the following matrix form:
\begin{equation}\label{form11}
A_x=\begin{pmatrix}
0&0&a_x&b_x\\
0&0&-a_x&-b_x\\
0&0&c_x&d_x\\
0&0&h_x&f_x
\end{pmatrix}\quad a_x,b_x,c_x,d_x,f_x,h_x \in \mathbb{C}
\end{equation}
We write $\Delta(x)=B\bar{x}$, where $B=(b_{ij})_{1\leq i,j\leq 4}, \, \bar{x}=(x_1,x_2,x_3,x_4)$ is the vector correspond to $x$. Then, $\Delta(x)=D_x(x)$ implies that
\begin{equation}
\begin{cases}
\,b_{11}x_1+b_{12}x_2+b_{13}x_3+b_{14}x_4=a_xx_3+b_xx4,&\\
\,b_{21}x_1+b_{22}x_2+b_{23}x_3+b_{24}x_4=-a_xx_3-b_xx_4,&\\
\,b_{31}x_1+b_{32}x_2+b_{33}x_3+b_{34}x_4=c_xx_3+d_xx_4,&\\
\,b_{41}x_1+b_{42}x_2+b_{43}x_3+b_{44}x_4=h_xx_3+f_xx_4.
\end{cases}
\end{equation} 
Which implies that
\begin{equation}
\begin{cases}
\,b_{11}=b_{12}=b_{21}=b_{22}=b_{31}=b_{32}=b_{41}=b_{42},&\\
\, b_{13}=-b_{23},&\\
\,b_{14}=-b_{24},&\\
\, b_{33}=c_x,&\\
\,b_{34}=d_x,&\\
\, b_{43}=h_x,&\\
\, b_{44}=f_x.
\end{cases}
\end{equation}
These equalities show that the matrix $B$ of the linear map $\Delta$ is of the form (\ref{form11}).\\
Therefore, $\Delta$ is a derivation. This completes the proof.
\end{proof}
Similar guments can be applied to the remaining $4$-dimensional $\omega$-Lie algebras, so we summarize the result in the following Theorem, without detailed proofs.
\begin{theorem}
Every local derivation of any $4$-dimensional $\omega$-Lie algebras over $\mathbb{C}$ is a derivation.
\end{theorem}
In \cite{oubb} we are prove that any local derivation of three dimensional $\omega$-Lie algebra over $\mathbb{C}$ is a derivation. Since, any local derivations on finite dimensional complex semisimple Lie algebra is a derivation (see \cite{Ayo2}), then the following theorem hold.
\begin{theorem}
 Every local derivation of any finite-dimensional semisimple $\omega$-Lie algebras over $\mathbb{C}$ is a derivation.   
\end{theorem}
Now, we prove that any 2-local derivation on 4-dimensional $\omega$-Lie algebra over $\mathbb{C}$ is a derivation.
\begin{theorem}
Every $2$-Local derivation of the $\omega$-algebras $L_{1,3}$, $L_{1,4}$, $L_{1,7}$,  $L_{1,8},$  $L_{2,2},$ $L_{2,3}$,  $L_{2,4}$, $F_{1,\alpha}$ is a derivation.
\end{theorem}
\begin{proof}
The proof of the theorem will be presented for $L_{1,4}$, similar demonstrations can be carried out for the remaining cases.\\
Let $\Delta$ be an arbitrary $2$local derivation of $L_{1,4}$,then by definition, for every element $x \in L_{1,4}$, there exists a derivation $D_{x,e_3}$ of $L_{1,4}$ such that 
$$\Delta(x)=D_{x,e_3}(x)\quad and \quad \Delta(e_3)=D_{x,e_3}(e_3).$$
By Table.3 in \cite{Chen2}, the matrix $A_{x,e_3}$ of the derivation $D_{x,e_3}$ has the following matrix form
\begin{equation}\label{14}
A_{x,e_3}=\begin{pmatrix}
0&0&a_{x,e_3}&0\\
0&0&-a_{x,e_3}&0\\
0&0&a_{x,e_3}&0\\
0&0&b_{x,e_3}&0
\end{pmatrix}
\end{equation}
Let $y$ be an arbitrary element in $L_{1,4}$. Then there existe a derivation $D_{y,e_3}$ of $L_{1,4}$ such that 
$$\Delta(y)=D_{y,e_3}(y)\quad and \quad \Delta(e_3)=D_{y,e_3}(e_3).$$ 
By Table.3 in \cite{Chen2}, the matrix $A_{y,e_3}$ of the derivation $D_{y,e_3}$ has the following matrix form
\begin{equation}\label{14}
A_{y,e_3}=\begin{pmatrix}
0&0&a_{y,e_3}&0\\
0&0&-a_{y,e_3}&0\\
0&0&a_{y,e_3}&0\\
0&0&b_{y,e_3}&0
\end{pmatrix}
\end{equation}
Since $\Delta(e_3)=D_{x,e_3}(e_3)=D_{y,e_3}(e_3)$, we have
$$a_{x,e_3}=a_{y,e_3} \quad and \quad b_{x,e_3}=b_{y,e_3}.$$
That it
$$D_{x,e_3}=D_{y,e_3}$$
Therefore, for any element $x$ of the algebra  $L_{1,4}$
$$\Delta(x)=D_{y,e_3}(x),$$
that it $D_{y,e_3}$ does not depend on $x$. Hence, $\Delta$ is a derivation of $L_{1,4}$.
\end{proof}
In \cite{oubb} we are prove that any 2-local derivation of three dimensional $\omega$-Lie algebra over $\mathbb{C}$ is a derivation. Since, any 2-local derivations on finite dimensional complex semisimple Lie algebra is a derivation (see \cite{Ayo3}), then the following theorem hold.
\begin{theorem}
  Every $2$-Local derivation on finite-dimensional semisimple $\omega$-algebras over $\mathbb{C}$ is a derivation.  
\end{theorem}

\section{Local and $2$-local automorphisms}
The examination of local and $2$-local automorphisms of $\omega$-Lie algebras $\mathfrak{g}$ heavily depends on a pivotal theorem established by Chen Y. et al., as detailed in their publication \cite{Chen2}. This theorem serves as our primary tool for delving into the complexities of the relevant automorphisms.
\begin{theorem}
Every local automorphism of $4$-dimensional $\omega$-Lie algebra over $\mathbb{C}$ is an automorphism.
\end{theorem}
\begin{proof}
Let $\psi$ be an arbitrary local automorphism of $L_{1,1}$ and $B=(b_{ij})_{1\leq i,j \leq 4}$ its matrix, i.e.,
$$\psi(x)=B\bar{x},\, x\in L_{1,1},$$
where $\bar{x}$ is the vector corresponding to $x$. Then by definition, for every element $x \in L_{1,1}$, there exist elements $a_x\neq -1,b_x$ in $\mathbb{C}$ such that
$$A_x=\begin{pmatrix}
1&0&a_x&0\\
0&1&-a_x&0\\
0&0&a_x+1&0\\
0&0&b_x&1
\end{pmatrix}$$
and 
$$\psi(x)=B\bar{x}=A_x\bar{x}.$$
Therefore, 
\begin{equation}
\begin{cases}
\,b_{11}x_1+b_{12}x_2+b_{13}x_3+b_{14}x_4=x_1+a_xx_3,&\\
\,b_{21}x_1+b_{22}x_2+b_{23}x_3+b_{24}x_4=x_2-a_xx_3,&\\
\,b_{31}x_1+b_{32}x_2+b_{33}x_3+b_{34}x_4=(a_x+1)x_3,&\\
\,b_{41}x_1+b_{42}x_2+b_{43}x_3+b_{44}x_4=b_xx_3+x_4.
\end{cases}
\end{equation}
Then 
\begin{equation}
\begin{cases}
\, b_{12}=b_{14}=b_{21}=b_{24}=b_{31}=b_{32}=b_{34}=b_{41}=b_{42}=0,&\\
\,b_{11}=b_{22}=b_{44}=1,&\\
\, b_{13}=-b_{23}=b_{33}-1=a_x,&\\
\, b_{43}=b_x.
\end{cases}
\end{equation}
Hence, by Table.4 in \cite{Chen2}, $\psi$ is an automorphism of $L_{1,1}$. The proof of the other cases is similar.
\end{proof}
In \cite{oubb} we are prove that any local qutomorphism of three dimensional $\omega$-Lie algebra over $\mathbb{C}$ is an automorphism. Since, any local automorphisms on finite dimensional complex simple Lie algebra is an automorphis or an anti-automprphis (see \cite{Cos}), then the following theorem hold.
\begin{theorem}
   Every local automorphism on finite-dimensional simple $\omega$-algebras over $\mathbb{C}$ is an automorphism or an anti-automorphism. 
\end{theorem}
\begin{theorem}
Every $2$-Local automorphism of the $\omega$-algebras $L_{1,3}$, $L_{1,4}$, $L_{1,7}$,  $L_{1,8},$  $L_{2,2},$ $L_{2,3}$,  $L_{2,4}$, $F_{1,\alpha}$ is an automorphism.
\end{theorem}
\begin{proof}
The proof of the theorem will be presented for $L_{2,3}$, similar demonstrations can be carried out for the remaining cases.\\
Let $\psi$ be an arbitrary $2$local automorphism of $L_{2,3}$,then by definition, for every element $x \in L_{2,3}$, there exists a derivation $\psi_{x,e_4}$ of $L_{2,3}$ such that 
$$\psi(x)=\psi_{x,e_4}(x)\quad and \quad \psi(e_4)=\psi_{x,e_4}(e_4).$$
By Table.4 in \cite{Chen2}, the matrix $A_{x,e_4}$ of the automorphism $\psi_{x,e_4}$ has the following matrix form
\begin{equation}\label{23}
A_{x,e_4}=\begin{pmatrix}
1&0&0&a_{x,e_4}\\
0&1&0&-a_{x,e_4}\\
0&0&1&a_{x,e_4}\\
0&0&0&b_{x,e_4}
\end{pmatrix} \quad b_{x,e_4} \neq 0,
\end{equation}
Let $y$ be an arbitrary element in $L_{2,3}$. Then there existe a automorphism $\psi_{y,e_4}$ of $L_{2,3}$ such that 
$$\psi(y)=\psi_{y,e_4}(y)\quad and \quad \psi(e_4)=\psi_{y,e_4}(e_4).$$ 
By Table.4 in \cite{Chen2}, the matrix $A_{y,e_4}$ of the derivation $\psi_{y,e_4}$ has the following matrix form
\begin{equation}\label{23'}
A_{y,e_4}=\begin{pmatrix}
1&0&0&a_{y,e_4}\\
0&1&0&-a_{y,e_4}\\
0&0&1&a_{y,e_4}\\
0&0&0&b_{y,e_4}
\end{pmatrix} \quad b_{x,e_4} \neq 0,
\end{equation}
Since $\psi(e_4)=\psi_{x,e_4}(e_4)=\psi_{y,e_4}(e_4)$, we have
$$a_{x,e_4}=a_{y,e_4} \quad and \quad b_{x,e_4}=b_{y,e_4}\neq 0.$$
That it
$$\psi_{x,e_3}=\psi_{y,e_3}$$
Therefore, for any element $x$ of the algebra  $L_{2,3}$
$$\psi(x)=\psi_{y,e_4}(x),$$
that it $\psi_{y,e_4}$ does not depend on $x$. Hence, $\psi$ is an automorphism of $L_{2,3}$.
\end{proof}
In \cite{oubb} we are prove that any 2-local qutomorphism of three dimensional $\omega$-Lie algebra over $\mathbb{C}$ is an automorphism. Since, any 2-local automorphisms on finite dimensional complex simple Lie algebra is an automorphis (see \cite{loc}), then the following theorem hold.
\begin{theorem}
 Every 2-local automorphism on finite-dimensional simple $\omega$-algebras over $\mathbb{C}$ is an automorphism.     
\end{theorem}

\section{Biderivations of $\omega$-Lie algebras}
In this section, we provide characterizations of biderivations for $4$-dimensional $\omega$-Lie algebras. We furnish detailed proofs for some cases, while omitting the proofs for the remaining cases due to the similarity of arguments.\\

We commence by revisiting the definition of a biderivation in the context of an arbitrary Lie algebra.
 \begin{definition}
 \label{def5}
 Let $(\mathfrak{g},[,])$ be an arbitrary algebra. A bilinear map $\delta : \mathfrak{g} \times \mathfrak{g} \rightarrow \mathfrak{g}$ is called a biderivation on $\mathfrak{g}$ if
 $$\delta([x,y],z)=[x,\delta(y,z)]+[\delta(x,z),y],$$
 $$\delta(x,[y,z])=[y,\delta(x,z)]+[\delta(x,y),z],$$
 for all $x,y,z \in \mathfrak{g}$. Denote by $BDer(\mathfrak{g})$ the set of all biderivations on $\mathfrak{g}$ which is clearly a vector space.
 \end{definition}
  
  Let $BDer(\mathfrak{g})$ represent the collection of all biderivations defined on $\mathfrak{g}$. It is evident that this set forms a vector space.\\
  A biderivation $\delta \in BDer(\mathfrak{g})$ is termed symmetric if $\delta(x,y)=\delta(y,x)$ holds for all $x,y \in \mathfrak{g}$; conversely, it is designated as skew-symmetric if $\delta(x,y)=-\delta(y,x)$ for all $x,y \in \mathfrak{g}$. We denote the subspaces of all symmetric biderivations and all skew-symmetric biderivations on $\mathfrak{g}$ as $BDer_+(\mathfrak{g})$ and $BDer_-(\mathfrak{g})$, respectively.
  \begin{definition}
  Let $\mathfrak{g}$ be a finite-dimensional $\omega$-Lie algebra. A bideivation $\delta \in BDer(\mathfrak{g}$ is called $\omega$-bideivation of $\mathfrak{g}$ if the linear maps $\delta(x,.)$ and $\delta(.,x)$ are a $\omega$-derivation. That is,
  \begin{eqnarray*}
  \omega(\delta(x,y),z)&=&\omega(y,\delta(x,z)),\\
  \omega(\delta(x,y),z)&=&\omega(x,\delta(z,y)).
  \end{eqnarray*}
  for all $x,y,z \in \mathfrak{g}$
  \end{definition}
  We write $BDer_\omega(\mathfrak{g})$ for the set of all $\omega$-biderivations of $\mathfrak{g}$. It is easy to see that $BDer_\omega(\mathfrak{g}) \subseteq BDer(\mathfrak{g})$.

For any $\delta \in BDer(\mathfrak{g})$, we define two bilinear maps by 
$$\delta^+(x,y)=\frac{1}{2}(\delta(x,y)+\delta(y,x)), \hspace{0.3 cm}\delta^-(x,y)=\frac{1}{2}(\delta(x,y)-\delta(y,x))$$
It is easy to see $\delta^+ \in BDer_+(\mathfrak{g})$ and $\delta^- \in BDer_-(\mathfrak{g})$. Since $\delta =\delta^+ + \delta^-$ it follows that
$$BDer(\mathfrak{g})=BDer_+(\mathfrak{g}) \oplus BDer_-(\mathfrak{g})$$
To characterize $BDer(\mathfrak{g})$, we only need to characterize $BDer_+(\mathfrak{g})$ and $BDer_-(\mathfrak{g})$.\\

Now, let $\delta$ be a  biderivation on $\omega$-Lie algebra $\mathfrak{g}$ and $x,y \in \mathfrak{g}$, such that $x=\sum_{i=1}^{4}x_{i}e_{i}$ and $y=\sum_{i=1}^{4} y_{i}e_{i}$. Then, by the bilinearity of $\delta$, we obtain, 
\begin{equation}\label{bid}
\delta(x,y)=\sum_{i=1}^{4} \sum_{j=1}^{4} x_{i} y_{j} \delta(e_i,e_j)=\sum_{i=1}^{4} \sum_{j=1}^{4} x_{i} y_{j} \delta_{e_i}(e_j).
\end{equation}
\begin{theorem}
The $4$-dimensional $\omega$-Lie algebra over $\mathbb{C}$ has no nontrivial skew-symmetric biderivations.
\end{theorem}
\begin{proof} We shall establish the theorem for $L_{1,6}$; the remaining cases can be demonstrated analogously.
 Let $\delta^-$ be an arbitrary skew-symmetric biderivation on $L_{1,6}$. By Table.4 in  \cite{Chen2}, the matrix $D_{e_i}$ of $\delta_{e_i}$, for $i=1,2,3,4$ is of the form 
$$D_{e_i}=\begin{pmatrix}
-a_i&0&c_i&-2a_i\\
0&0&-c_i&a_i\\
0&0&b_i&0\\
0&0&c_i&-a_i 
\end{pmatrix}$$
Since $ \delta$ is skew-symmetric, then, the equalities $\delta(e_i,e_i)=0$  $i=1,2,3,4$ implies that
$$D_{e_1}=\begin{pmatrix}
0&0&c_1&0\\
0&0&-c_1&0\\
0&0&b_1&0\\
0&0&c_1&0 
\end{pmatrix}, \quad D_{e_2}=\begin{pmatrix}
-a_2&0&c_2&-2a_2\\
0&0&-c_2&a_2\\
0&0&b_2&0\\
0&0&c_2&-a_2 
\end{pmatrix}$$

$$D_{e_3}=\begin{pmatrix}
-a_3&0&0&-2a_3\\
0&0&0&a_3\\
0&0&0&0\\
0&0&0&-a_3 
\end{pmatrix},\quad D_{e_4}=\begin{pmatrix}
0&0&c_4&0\\
0&0&-c_4&0\\
0&0&b_4&0\\
0&0&c_4&0 
\end{pmatrix}.$$
The equalities $\delta^-(e_1,e_i)=-\delta^-(e_i,e_1)$ for $i=2,3,4$ implies that 
$$D_{e_1}=\begin{pmatrix}
0&0&0&0\\
0&0&0&0\\
0&0&0&0\\
0&0&0&0 
\end{pmatrix}, \quad D_{e_2}=\begin{pmatrix}
0&0&c_2&0\\
0&0&-c_2&0\\
0&0&b_2&0\\
0&0&c_2&0 
\end{pmatrix}$$

$$D_{e_3}=\begin{pmatrix}
0&0&0&0\\
0&0&0&0\\
0&0&0&0\\
0&0&0&0 
\end{pmatrix},\quad D_{e_4}=\begin{pmatrix}
0&0&c_4&0\\
0&0&-c_4&0\\
0&0&b_4&0\\
0&0&c_4&0 
\end{pmatrix}.$$
Finally, from the equalities $\delta^-(e_3,e_i)=-\delta^-(e_i,e_3)$ for $i=2,4$ we get
$$D_{e_1}=\begin{pmatrix}
0&0&0&0\\
0&0&0&0\\
0&0&0&0\\
0&0&0&0 
\end{pmatrix}, \quad D_{e_2}=\begin{pmatrix}
0&0&0&0\\
0&0&0&0\\
0&0&0&0\\
0&0&0&0 
\end{pmatrix}$$

$$D_{e_3}=\begin{pmatrix}
0&0&0&0\\
0&0&0&0\\
0&0&0&0\\
0&0&0&0 
\end{pmatrix},\quad D_{e_4}=\begin{pmatrix}
0&0&0&0\\
0&0&0&0\\
0&0&0&0\\
0&0&0&0 
\end{pmatrix}.$$
Let $x=x_1e_1+x_2e_2+x_3e_3+x_4e_4$, $y=y_1e_1+y_2e_2+y_3e_3+y_4e_4$, then
$$\delta^-(x,y)=\sum_{i=1}^4x_i\delta^-(e_i,y)=\sum_{i=1}^4x_iD_{e_i}\bar{y}=0$$. Which ends the proof.
\end{proof}
\begin{theorem}
$\delta^+$ is symmetric biderivation on $L_{1,1}$ if and only if there exist the elements $a,b,c,d,f,g,h,m,n \in \mathbb{C}$ such that
\begin{eqnarray*}
\delta^+(x,y)&=&\Big(x_3(ay_3+by_4)+x_4(y_3+gy_4)\Big)e_1-\Big(x_3(ay_3+by_4+by_3+gy_4)\Big)e_2\\
&&+\Big(x_3(cy_3+dy_4)+x_4(dy_3+my_4)\Big)e_3+\Big(x_3(hy_3+fy_4)+x_4(fy_3+ny_4)\Big)e_4
\end{eqnarray*}
for all $x=\sum_{i=1}^4x_ie_i, \, y=\sum_{i=1}^4y_je_j \in L_{1,1}$.
\end{theorem}
\begin{proof}
Let $\delta^+$ be an arbitrary symmetric derivation on $L_{1,1}$. By Table.4 in  \cite{Chen2}, the matrix $D_{e_i}$ of $\delta_{e_i}$, for $i=1,2,3,4$ is of the form 
$$D_{e_i}=\begin{pmatrix}
0&0&a_i&b_i\\
0&0&-a_i&-b_i\\
0&0&c_i&d_i\\
0&0&h_i&f_i 
\end{pmatrix}$$
Since $\delta^+(e_1,e_i)=\delta^+(e_i,e_1)$ for $i=2,3,4$ Then, 
$$D_{e_1}=\begin{pmatrix}
0&0&0&0\\
0&0&0&0\\
0&0&0&0\\
0&0&0&0 
\end{pmatrix}, \quad D_{e_2}=\begin{pmatrix}
0&0&0&0\\
0&0&0&0\\
0&0&0&0\\
0&0&0&0 
\end{pmatrix}$$
From the equality $\delta^+(e_3,e_4)=\delta^+(e_4,e_3)$ we deduce
$$b_3=a_4, \quad d_3=c_4, \quad f_3=h_4.$$
Set $b_3=a_4=b,\, d_3=c_4=d, \, f_3=h_4=f, \, g=b_4,\, m=d_4, \, n=f_4$. Then,
$$D_{e_1}=\begin{pmatrix}
0&0&0&0\\
0&0&0&0\\
0&0&0&0\\
0&0&0&0 
\end{pmatrix}, \quad D_{e_2}=\begin{pmatrix}
0&0&0&0\\
0&0&0&0\\
0&0&0&0\\
0&0&0&0 
\end{pmatrix}\\
D_{e_3}=\begin{pmatrix}
0&0&a&b\\
0&0&-a&-b\\
0&0&c&d\\
0&0&h&f 
\end{pmatrix},\quad D_{e_4}=\begin{pmatrix}
0&0&b&g\\
0&0&-b&-g\\
0&0&d&m\\
0&0&f&n 
\end{pmatrix}.$$
Therefor, for all $x=\sum_{i=1}^4x_ie_i$, $y=\sum_{j=1}^4y_je_j\in L_{1,1}$ we have,
\begin{eqnarray*}
\delta^+(x,y)\, =\, \sum_{i=1}^4x_i\delta^+_{e_i}(y)&=&\sum_{i=1}^4x_iD_{e_i}\bar{y}\\
&=&\Big(x_3(ay_3+by_4)+x_4(y_3+gy_4)\Big)e_1-\Big(x_3(ay_3+by_4+by_3+gy_4)\Big)e_2\\
&&+\Big(x_3(cy_3+dy_4)+x_4(dy_3+my_4)\Big)e_3+\Big(x_3(hy_3+fy_4)+x_4(fy_3+ny_4)\Big)e_4.
\end{eqnarray*}
\end{proof}
Similar argements can be applied to the remaining $4$-dimensional $\omega$-Lie algebras, so we summarize the result in the following Table $1$, without detailed proofs
\begin{center}
Table $1$ : Symmetric biderivation of $4$-dimensional $\omega$-Lie algebras.\\
\begin{tabular}{|c|c|}\hline
$\mathfrak{g}$& elements in $Bder_+(\mathfrak{g})$ \\ \hline
$L_{1,1}$&$\Big(x_3(ay_3+by_4)+x_4(y_3+gy_4)\Big)e_1-\Big(x_3(ay_3+by_4+by_3+gy_4)\Big)e_2$\\&$+\Big(x_3(cy_3+dy_4)+x_4(dy_3+my_4)\Big)e_3+\Big(x_3(hy_3+fy_4)+x_4(fy_3+ny_4)\Big)e_4$\\ &$a,b,c,d,f,g,h,m,n \in \mathbb{C}$\\\hline
$L_{1,2}$& $\Big(x_3(ay_3+by_4)+bx_4y_3\Big)e_1+\Big(x_3((b-a)y_3-by_4)-bx_4y_3\Big)e_2$\\&$+\Big(cx_3y_3\Big) e_3+\Big(x_3(dy_3+cy_4)+cx_4y_3\Big)e_4,\,  a,b,c,d \in \mathbb{C}$\\\hline
$L_{1,3}$&$ax_3y_3e_3+bx_3y_3e_4, $\\&$ \,a,b \in \mathbb{C}$\\\hline 
$L_{1,4}$&$ax_3y_3e_1-ax_3y_3e_2+ax_3y_3e_3+bx_3y_3e_4,$\\&$ \,a,b \in \mathbb{C} $\\\hline
$L_{1,5}$&$(ax_3y_3-2dx_4y_4)e_1+(-ax_3y_3+dx_4y_4)e_2+bx_3y_3e_3+cx_4y_4e_4,\,$\\&$ a,b,c,d \in \mathbb{C} $\\ \hline
$L_{1,6}$& $ax_3y_3e_1-ax_3y_3e_2+bx_3y_3e_3+ax_3y_4e_4,$\\&$\,a,b \in \mathbb{C}$\\ \hline
$L_{1,7}$& $\Big(x_3(ay_3+by_4)+bx_4(y_3+2y_4)\Big)e_1+\Big(x_3((b-a)y_3-\frac{b}{2}y_4)+bx_4(y_3-y_4)\Big)e_2,$\\&$ \, a,b \in \mathbb{C}$\\ \hline
$L_{1,8}$& $ax_3y_3e_1-ax_3y_3e_2+ax_3y_3e_4,\,$\\&$ a \in \mathbb{C}$\\ \hline
$L_{2,1}$&$ax_4y_4e_3+bx_4y_4e_4,$\\&$ \,a,b \in \mathbb{C}$\\ \hline
$L_{2,2}$&$0$\\ \hline
$L_{2,3}$&$ax_4y_4e_1-ax_4y_4e_2+ax_4y_4e_3+bx_4y_4e_4,$\\&$\,a,b \in \mathbb{C}$ \\ \hline
$ L_{2,4}$ &$0$ \\ \hline
$E_{1,\alpha}$ \\($\alpha \neq 0,1$)&$\Big(x_3(ay_3+\alpha(\alpha+1)y_4)+x_4(\alpha(\alpha+1)y_3-(\alpha+1)cy_4)\Big)e_1+\Big(x_3((b-\alpha)y_3-\alpha y_4)+x_4(-\alpha y_3+cy_4)\Big)e_2$\\&$+\Big(bx_3y_3+dx_4y_4\Big)e_4,\, a,b,c,d \in \mathbb{C}$ \\ \hline
$F_{1,\alpha}$ \\($\alpha \neq 0,1$)& $\alpha a x_3y_3e_1-\alpha ax_3y_3e_2+bx_3y_3e_3+ax_3y_3e_4,$\\&$ \,a,b \in \mathbb{C}$ \\ \hline
$ G_{1, \alpha}$&$0$ \\ \hline
$,H_{1,\alpha}$&$0$\\ \hline
$\tilde{A_\alpha}$& $ax_4y_4e_4,\, a \in \mathbb{C}$ \\ \hline
$\tilde{B}$&$ax_2y_2e_3+bx_4y_4e_4,$\\&$\,a,b \in \mathbb{C}$\\ \hline
$\tilde{C_\alpha}$ \\ ($\alpha \neq 0,-1,1$)&$ax_4y_4e_4,$\\&$ \, a \in \mathbb{C}$ \\ \hline
$\tilde{C_1}$&$\Big(x_2(ay_2+cy_3)+x_3(cy_2+by_3)\Big)e_2+\Big(x_2(dy_2-ay_3)-x_3(ay_2+cy_3)\Big)e_3+fx_4y_4e_4,$\\&$\,a,b,c,d,f \in \mathbb{C}$\\ \hline
\end{tabular}
\end{center}
It is know that (see \cite{tang}) $\delta$ is a biderivation on complex finite-dimensional simple Lie algebra $\mathfrak{g}$ if and only if there exists $\lambda \in \mathbb{C}$ such that $\delta(x,y)=\lambda[x,y]$, for all $x,y \in \mathfrak{g}$. In \cite{oubb} we are show that any biderivations of $A_\alpha (\alpha \in \mathbb{C})$, $B$, $C_\alpha (0,-1\neq \alpha \in \mathbb{C}$ is the zero mapping. Then, we have the following characterization of biderivation of biderivation of finite-dimensional $\omega$-Lie algebra over $\mathbb{C}$.
\begin{theorem}
    Let $\mathfrak{g}$ be finite-dimensional simple $\omega$-Lie algebra over $\mathbb{C}$. Then, $\delta:\mathfrak{g}\times \mathfrak{g}\to \mathfrak{g}$ is a biderivation if and only if there exists $\lambda\in \mathbb{C}$ such that
    $$\delta(x,y)=\lambda[x,y],\quad \forall x,y \in \mathfrak{g}.$$
\end{theorem}

We concluding the results on the paper we present the foollowing theorem.
\begin{theorem}
$BDer_\omega(\mathfrak{g})=BDer(\mathfrak{g})$ for any $4$-dimensional nontrivial $\omega$-Lie algebra $\mathfrak{g}$, except for $\mathfrak{g}=L_{1,6}$ and $L_{1,8}$.
\end{theorem}
\begin{proof}
Let $\delta \in BDer(\mathfrak{g})$ where $\mathfrak{g} \neq L_{1,6},L_{1,8}$, then for all $x \in \mathfrak{g}$, $\delta(x,.)$ (resp. $\delta(.,x)$) is a derivation of $\mathfrak{g}$. Then, by Proposition 1.5 of \cite{Chen0} we have $\delta(x,.)$ (resp. $\delta(.,x)$) is an $\omega$-derivation. Therefore, $\delta$ is an $\omega$-bedirivation of $\mathfrak{g}$.\\
We prove now that $BDer_\omega(L_{1,6}) \neq BDer(L_{1,6})$. Consider a biderivation $\delta \in BDer(L_{1,6})$ defined by 
$$\delta(x,y)=ax_3y_3e_1-ax_3y_3e_2+bx_3y_3e_3+ax_3y_4e_4,\, \forall x=\sum_{i=1}^4x_ie_i, y=\sum_{j=1}^4y_je_j \in L_{1,6}.$$
with $b \neq 0$. Then
$$\omega(\delta(e_3,e_4),e_2)=\omega(be_1,e_2)=b\quad and \omega(e_3,\delta(e_2,e_2))=0.$$
Thus, $\delta$ is not an $\omega$-biderivation.\\
Similar arguments show that $BDer_\omega(L_{1,8}) \neq BDer(L_{1,8})$.
\end{proof}
\section{$\frac{1}{2}$-derivations of simple $\omega$-Lie algebras}
The description of $\frac{1}{2}$-derivations of simple finite-dimensional Lie algebras is given in \cite{13}. There are no non-trivial $\frac{1}{2}$-derivations of simple finite-dimensional Lie algebras. In this section we give the description of $\frac{1}{2}$-derivations of simple finite-dimensional $\omega$-Lie algebras.
\begin{definition}
Let $(\mathfrak{g},[,])$ be an $\omega$-Lie algebra. A linear map $\Delta : \mathfrak{g} \rightarrow \mathfrak{g}$ is called $\frac{1}{2}$-derivation, if 
$$\Delta([x,y])=\frac{1}{2}\Big([\Delta(x),y]+[x,\Delta(y)]\Big),$$
for all $x,y \in \mathfrak{g}$.
\end{definition}
\begin{definition}
Let $(\mathfrak{g},[,])$ be an $\omega$-Lie algebra. A linear map $f : \mathfrak{g} \rightarrow \mathfrak{g}$ is called local $\frac{1}{2}$-derivation, if for all $x\in \mathfrak{g}$ there exists an $\frac{1}{2}$-derivation $\Delta_x$ (depending on $x$) of $\mathfrak{g}$ such that
$f(x)=\Delta_x(x).$
\end{definition}
\begin{definition}
Let $(\mathfrak{g},[,])$ be an $\omega$-Lie algebra. A  map (not necessarily linear)  $f : \mathfrak{g} \rightarrow \mathfrak{g}$ is called $2$-local $\frac{1}{2}$-derivation, if for all $x,y\in \mathfrak{g}$ there exists an $\frac{1}{2}$-derivation $\Delta_{x,y}$ of $\mathfrak{g}$ such that
$$f(x)=\Delta_{x,y}(x)\quad \mbox{and} \quad f(y)=\Delta_{x,y}(y).$$
\end{definition}
\begin{theorem}
    Let $(\mathfrak{g},[,])$ be a complex simple $\omega$-Lie algebra. Then, $\Delta$ is an $\frac{1}{2}$-derivation of $\mathfrak{g}$ if and only if there exists $\lambda \in \mathbb{C}$ such that $\Delta(x)=\lambda x$ for all $x \in \mathfrak{g}$.
\end{theorem}
\begin{proof}
    Let  $(\mathfrak{g},[,])$ be a complex simple $\omega$-Lie algebra and $\Delta$ an $\frac{1}{2}$-derivation of $\mathfrak{g}$.\\
    \textbf{Cases: 1}  $\mathfrak{g}$ is a Lie algebra, then, by \cite{13} any $\frac{1}{2}$-derivation of $\mathfrak{g}$ is trivial.\\
    \textbf{Cases: 2}  $\mathfrak{g}=A_\alpha$, we set $\Delta(e_j)=\sum_{i=1}^3a_{ij}e_i$ for $j=1,2,3$.\\
    The equality $2\Delta(e_1)=[\Delta(e_1),e_2]+[e_1,\Delta(e_2)]$ impies that
    \begin{equation}\label{a1}
     2a_{11}=a_{11}-\alpha a_{31}+a_{22}, \quad 2a_{21}=a_{32}, \quad 2a_{31}=-a_{31}.   
    \end{equation}
    The equality $2\Big(\Delta(e_1)+\Delta(e_2)\Big)=[\Delta(e_1),e_3]+[e_1,\Delta(e_3)]$ impies that
    \begin{equation}\label{a2}
     2a_{11}+2a_{12}=a_{11}+\alpha a_{21}+a_{23}+a_{33}, \quad 2a_{21}+2a_{22}=a_{11}+a_{33}, \quad 2a_{31}+2a_{32} =a_{21}.   
    \end{equation}
     The equality $2\Big(\Delta(e_3)+\alpha\Delta(e_1)\Big)=[\Delta(e_2),e_3]+[e_2,\Delta(e_3)]$ impies that
    \begin{equation}\label{a3}
     2a_{13}+2\alpha a_{11}=\alpha a_{22}-a_{13}+\alpha a_{33}, \quad 2a_{23}+2\alpha a_{21}=a_{12}, \quad 2a_{33}+2\alpha a_{31} =a_{12}+a_{22}+a_{33}.   
    \end{equation}
    By comparing equations (\ref{a1}), (\ref{a2}) and (\ref{a3}) we get $\Delta(e_i)=\lambda e_i$ where $\lambda=a_{11}=a_{22}=a_{33}$.
     \textbf{Cases: 3}  $\mathfrak{g}=B$, we set $\Delta(e_j)=\sum_{i=1}^3a_{ij}e_i$ for $j=1,2,3$.\\
    The equality $2\Delta(e_2)=[\Delta(e_1),e_2]+[e_1,\Delta(e_2)]$ impies that
    \begin{equation}\label{b1}
     2a_{11}=- a_{31}, \quad 2a_{22}=a_{11}+a_{22}+a_{32}, \quad 2a_{32}=a_{32}.   
    \end{equation}
    The equality $2\Delta(e_1)=[\Delta(e_2),e_3]+[e_2,\Delta(e_3)]$ impies that
    \begin{equation}\label{b2}
     2a_{11}=a_{22}+a_{33}, \quad 2a_{21}=a_{12}-a_{13}, \quad 2a_{31} =a_{12}.   
    \end{equation}
     The equality $2\Big(\Delta(e_2)+\Delta(e_3)\Big)=[\Delta(e_1),e_3]+[e_1,\Delta(e_3)]$ impies that
    \begin{equation}\label{b3}
     2a_{12}+2 a_{13}= a_{21}, \quad 2a_{22}+2 a_{23}=a_{11}+a_{23}+a_{33}, \quad 2a_{32}+2 a_{33} =a_{11}+a_{33}.   
    \end{equation}
    By comparing equations (\ref{b1}), (\ref{b2}) and (\ref{b3}) we get $\Delta(e_i)=\lambda e_i$ where $\lambda=a_{11}=a_{22}=a_{33}$.
     \textbf{Cases 4:}  $\mathfrak{g}=C_\alpha\, (\alpha \neq -1,0)$, we set $\Delta(e_j)=\sum_{i=1}^3a_{ij}e_i$ for $j=1,2,3$.\\
    The equality $2\Delta(e_2)=[\Delta(e_1),e_2]+[e_1,\Delta(e_2)]$ impies that
    \begin{equation}\label{c1}
     2a_{12}=-a_{31}, \quad 2a_{22}=a_{11}+a_{22}, \quad 2a_{32}=\alpha a_{32}.   
    \end{equation}
    The equality $2\alpha\Delta(e_3)=[\Delta(e_1),e_3]+[e_1,\Delta(e_3)]$ impies that
    \begin{equation}\label{c2}
     2a_{13}=a_{21}, \quad 2 \alpha a_{23}=a_{23}, \quad 2\alpha a_{33} =\alpha a_{11}+\alpha a_{33}.   
    \end{equation}
     The equality $2\Delta(e_1)=[\Delta(e_2),e_3]+[e_2,\Delta(e_3)]$ impies that
    \begin{equation}\label{c3}
     2 a_{11}=a_{22}+ a_{33}, \quad 2a_{21}=-a_{13}, \quad 2a_{31} =\alpha a_{12}.   
    \end{equation}
    By comparing equations (\ref{c1}), (\ref{c2}) and (\ref{c3}) we get $\Delta(e_i)=\lambda e_i$ where $\lambda=a_{11}=a_{22}=a_{33}$. The proof is completed.
\end{proof}
\begin{theorem}
    Every local  derivation of complex simple $\omega$-algebra  is a derivation.
\end{theorem}
\begin{proof}
   Let $f$ be a local $\frac{1}{2}$-derivation of $\mathfrak{g}$, then for any $x\in \mathfrak{g}$ there exists an $\frac{1}{2}$-derivatinon $\Delta_x$ of $\mathfrak{g}$ such that $f(x)=\Delta_x(x)$, then by Theorem \ref{th} there exists $\lambda_x \in \mathbb{C}$ such that $f(x)=\lambda_x x$. Let $B=\lbrace e_1,e_2,\dots,e_n\rbrace$ be a basise of $\mathfrak{g}$, then the matrix $M$ of $f$ has the following form $M=diag(\lambda_1,\lambda_2,\dots,\lambda_n)$. In other hand, $\lambda_{i+j}(e_i+e_j)=f(e_i+e_j)=f(e_i)+f(e_j)=\lambda_ie_i+\lambda_je_j$, therefore, $\lambda_i=\lambda_{i+j}=\lambda_j$. Thus, $M=\lambda I$. Which ends the proof.
\end{proof}
\begin{theorem}
    Every  $2$-local derivation of complex simple $\omega$-algebra  is a derivation.
\end{theorem}
\begin{proof}
    Let $f$ be a $2$-local $\frac{1}{2}$-derivation of $\mathfrak{g}$, then for any $x,y\in \mathfrak{g}$ there exists an $\frac{1}{2}$-derivatinon $\Delta_{x,y}$ of $\mathfrak{g}$ such that $f(x)=\Delta_{x,y}(x)$ and $f(y)=\Delta_{x,y}(y)$, then by Theorem \ref{th} there exists $\lambda_{x,y} \in \mathbb{C}$ such that $f(x)=\lambda_{x,y} x$ and $f(y)=\lambda_{x,y} y$. Let $z\in \mathfrak{g}$, then there exits $\lambda_{y,z}\in \mathbb{C}$ such that $f(z)=\lambda_{y,z} z$ and $f(y)=\lambda_{y,z} y$. Then, $\lambda_{x,y}y=f(y)=\lambda_{y,z}y$, so $\lambda_{x,y}=\lambda{y,z}$. Therefor, $f(x)=\lambda_{y,z}x$ i.e not depending on $x$. Therefore, $f$ is a $2$-local $\frac{1}{2}$-derivation.
\end{proof}
A relation between $\frac{1}{2}$-derivations of Lie algebras and transposed Poisson algebras has
been established by Ferreira et Al. in \cite{fer}. \textbf{Natural Question:} How can one define transposed Poisson structures on $\omega$-Lie algebras, and what is their relationship with $\frac{1}{2}$-derivations?
$$\noindent \textbf{Conclusion:}$$

The structure constants play a key role in our analysis. A computational challenge would be to develop an algorithm (using any computational tools such as Mathematica, Maple, or a computer
 algebra system like GAP) to systematically determine the biderivation and
local (2-local) derivations (automorphismes) in $\omega$-algebras of dimension $n \geq 5$, building upon the canonical forms proposed in this work.\\
$$\noindent \textbf{Acknowledgements:}$$ The authors thank the referees for their valuable comments that contributed to a sensible improvement of the paper.

\section*{Declarations}
\noindent \textbf{Funding Declaration:} The authors declare that no funds, grants, or other financial support were received during the preparation of this manuscript.\\
\noindent \textbf{Conflicts of interest:} The authors declare that they have no conflict of interest.\\
\noindent \textbf{Dataavailability:}No data was used for the research described in the article.


\begin{thebibliography}{99}
 
\bibitem{Ayo1}
Ayupov, Sh. Kudaybergenov, K.  2-Local automorphisms on finite-dimensional Lie algebras. Linear Algebra and its Applications 507 (2016) 121$-$131.

\bibitem{Ayo2}
Ayupov, Sh. Kudaybergenov, K. Local derivations on finite-dimensional Lie algebras. Linear Algebra and its Applications 493 (2016) 381$-$398.
\bibitem{Ayo3}
Ayupov, Shavkat, Karimbergen Kudaybergenov, and Isamiddin Rakhimov. "2-Local derivations on finite-dimensional Lie algebras." Linear Algebra and its Applications 474 (2015): 1-11.



\bibitem{Bresar}
Bre\v{s}ar, M., Martindale 3rd, W.S., Miers, C.R.:
 Centralizing maps in prime rings with involution. J. Algebra. 161, (1993) 342$-$357.
\bibitem{brz}
Bre\v{s}ar, M., Zhao, K.:
 Biderivations and  commuting Linear maps on  Lie algebras. J. Lie Theory. 28, (1018) 885$-$900.
 
 \bibitem{br}
Bre\v{s}ar, M.:
 Commuting maps: A survey. Taiwanese J. Math. 8, (2004)361$-$397.
 
 \bibitem{Chen0}
 Chen Y.,   Zhang R., Simple $\omega$-Lie algebras and $4$-dimensional $\omega$-Lie algebras over $\mathbb{C}$. To appear in Bull. Malays.
 \bibitem{Chen1}
  Chen Y., Liu C.,  Zhang R., Classification of three dimensional complex $\omega$-Lie algebras. Port. Math. (EMS) 71 (2014) $97-108$.
   \bibitem{loc}
   Chen, Zhengxin, and Dengyin Wang. "2-Local automorphisms of finite-dimensional simple Lie algebras." Linear Algebra and its Applications 486 (2015): 335-344.
   
   \bibitem{Chen2}
    Chen Y., Zhang, Z., Zhang, R., Zhuang, R.: Derivations, Automorphisms, and Representations of Complex $\omega$-Lie algebras. Commun. Algebra, 46(2),(2018) $708-726$.

    \bibitem{12}
    Chen, Yin, and Runxuan Zhang. "Simple $\omega$-Lie algebras and 4-dimensional $\omega$-Lie algebras over $\mathbb{C}$." Bulletin of the Malaysian Mathematical Sciences Society 40.3 (2017): 1377-1390.
\bibitem{Cos}
Costantini, Mauro. "Local automorphisms of finite dimensional simple Lie algebras." Linear Algebra and its Applications 562 (2019): 123-134.

\bibitem{Far}
  Farkas d.,  Letzter G., Ring theory from symplectic geometry. J. Pure Appl. Algebra 125 (1998), no. $1-3$, $155-190$.
  \bibitem{fer}
  Ferreira, Bruno Leonardo Macedo, Ivan Kaygorodov, and Viktor Lopatkin. "$\frac{1}{2}$-derivations of Lie algebras and transposed Poisson algebras." Revista de la Real Academia de Ciencias Exactas, Físicas y Naturales. Serie A. Matemáticas 115.3 (2021): 142.

  \bibitem{13}
  Filippov V., $\delta$-Derivations of prime Lie algebras, Siberian Mathematical Journal, 40 (1999), 1, 174–184.

 \bibitem{gle}
  Gleason A., A characterization of maximal ideals
J. Anal. Math., 19 (1967), pp. $171-172$.

\bibitem{kad}
  Kadison R., Local derivations J. Algebra, 130 (1990), pp. $494-509$.

\bibitem{kah}
Kahane J.,  Zelazko W., A characterization of maximal ideals in commutative Banach algebras
Stud. Math., 29 (1968), pp. $339-343$

  
  \bibitem{Mak}
 Maksa Gy., A remark on symmetric biadditive functions having nonnegative diagonalization. Glasnik Mat. Ser. III 15 (1980), no. 2, $279-282$.
 
 \bibitem{nur1}
   Nurowski P., Deforming a Lie algebra by means of a 2-form. J. Geom. Phys. 57 (2007) $1325-1329$.
  \bibitem{nur2}
   Nurowski P. , Distinguished dimensions for special Riemannian geometries. J. Geom. Phys. 58 (2008) $1148-1170$.


\bibitem{oub}
Oubba H., Generalized quaternion algebras, Rendiconti del Circolo Matematico di Palermo Series 2 Volume 72,(2023), pages $4239-4250$.
\bibitem{oubb}
Oubba, H. (2024). Local (2-Local) derivations and automorphisms and biderivations of complex $\omega$-Lie algebras. Le Matematiche, 79(1), 135-150.
\bibitem{sem}
  \v{S}emrl P.
Local automorphisms and derivations on 
Proc. Am. Math. Soc., 125 (1997), pp. $2677-2680$

\bibitem{Sko}
 Skosyrskii V., Strongly prime noncommutative Jordan algebras (Russian). Trudy Inst. Mat. (Novosibirsk) 16 (1989), $131-164$, $198-199$.
 \bibitem{tang}
 Tang, Xiaomin. "Biderivations of finite-dimensional complex simple Lie algebras." Linear and Multilinear Algebra 66.2 (2018): 250-259.
 \bibitem{Vuk}
 Vukman J., Symmetric bi-derivations on prime and semi-prime rings. Aequationes Math. 38 (1989), no. $2-3$, $245-254$.
\bibitem{11}
Zusmanovich, P.: $\omega$-Lie algebras. J. Geom. Phys. 60, 1028–1044 (2010)
\bibitem{Wang}
Wang, D., Yu, X., Chen, Z.:
 Biderivations of the parabolic subalgebra of simple Lie algebras, Comm. Algebra. 39,  (2011) 4097$-$4104
 
  

\end{thebibliography}
\end{document}